\documentclass[12pt,oneside,english]{amsart}
\usepackage[T1]{fontenc}
\usepackage[latin9]{inputenc}
\usepackage{geometry}
\usepackage{amsthm}
\usepackage{amstext}
\usepackage{amssymb}
\usepackage{kvsetkeys}

\makeatletter
\numberwithin{equation}{section}
\numberwithin{figure}{section}
\theoremstyle{plain}
\newtheorem{thm}{\protect\theoremname}[section]
  \theoremstyle{definition}
  \newtheorem{defn}[thm]{\protect\definitionname}
  \theoremstyle{plain}
  \newtheorem{prop}[thm]{\protect\propositionname}
  \theoremstyle{plain}
  \newtheorem{lem}[thm]{\protect\lemmaname}
  \theoremstyle{remark}
  \newtheorem*{note*}{\protect\notename}
  \theoremstyle{remark}
  \newtheorem*{acknowledgement*}{\protect\acknowledgementname}

\usepackage{amsmath,amsthm,amsfonts,amssymb,amscd,flafter,pinlabel}
\usepackage[mathscr]{eucal}
\usepackage[all,knot,arc]{xy}

\usepackage[usenames,dvipsnames]{color}
\usepackage{subfigure}
\usepackage{graphicx}
\usepackage{caption}
\usepackage{multirow}

\usepackage[colorlinks=true,linkcolor=blue,citecolor=blue,urlcolor=blue]{hyperref}

\makeatother

\usepackage{babel}
  \providecommand{\acknowledgementname}{Acknowledgement}
  \providecommand{\definitionname}{Definition}
  \providecommand{\lemmaname}{Lemma}
  \providecommand{\notename}{Note}
  \providecommand{\propositionname}{Proposition}
\providecommand{\theoremname}{Theorem}

\begin{document}

\title{The untwisting number of a knot}

\author{Kenan Ince}
\address{Department of Mathematics, Rice University, 
Houston, TX  77005}
\email{kenan@rice.edu}
\urladdr{http://math.rice.edu/~kai1} 

\dedicatory{Dedicated to Tim Cochran}
\begin{abstract}
The unknotting number of a knot is the minimum number of crossings
one must change to turn that knot into the unknot. The algebraic unknotting
number is the minimum number of crossing changes needed to transform
a knot into an Alexander polynomial-one knot. We work with a generalization
of unknotting number due to Mathieu-Domergue, which we call the untwisting
number. The untwisting number is the minimum number (over all diagrams
of a knot) of right- or left-handed twists on even numbers of strands
of a knot, with half of the strands oriented in each direction, necessary
to transform that knot into the unknot. We show that the algebraic
untwisting number is equal to the algebraic unknotting number. However,
we also exhibit several families of knots for which the difference
between the unknotting and untwisting numbers is arbitrarily large,
even when we only allow twists on a fixed number of strands or fewer.
\end{abstract}
\maketitle

\section{Introduction}

It is a natural knot-theoretic question to seek to measure ``how
knotted up'' a knot is. One such ``knottiness'' measure is given
by the \emph{unknotting number} \textbf{$u(K)$}, the minimum number
of crossings, taken over all diagrams of $K$, one must change to
turn $K$ into the unknot. By a \emph{crossing change }we shall mean
one of the two local moves on a knot diagram given in Figure \ref{fig:Crossing-changes}.

\begin{figure}
\def\svgwidth{0.5\columnwidth}
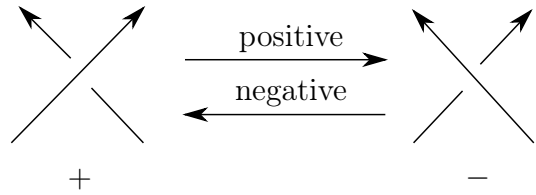

\protect\caption{\label{fig:Crossing-changes}A positive and negative crossing change.}
\end{figure}

This invariant is quite simple to define but has proven itself very
difficult to master. Fifty years ago, Milnor conjectured that the
unknotting number for the $(p,q)$-torus knot was $(p-1)(q-1)/2$;
only in 1993, in two celebrated papers \cite{kronheimer_gauge_1993,kronheimer_gauge_1995},
did Kronheimer and Mrowka prove this conjecture true. Hence, it is
desirable to look at variants of unknotting number which may be more
tractable. One natural variant (due to Murakami \cite{murakami_algebraic_1990})
is the \emph{algebraic unknotting number} $u_{a}(K)$, the minimum
number of crossing changes necessary to turn a given knot into an
Alexander polynomial-one knot. Alexander polynomial-one knots are
significant because they ``look like the unknot'' to \emph{classical
invariants}, knot invariants derived from the Seifert matrix. It is
obvious that $u_{a}(K)\leq u(K)$ for any knot $K$, and there exist
knots such that $u_{a}(K)<u(K)$ (for instance, any nontrivial knot
with trivial Alexander polynomial).

In \cite{mathieu_chirurgies_1988-1}, Mathieu and Domergue defined
another generalization of unknotting number. In \cite{livingston_slicing_2002},
Livingston worked with this definition. He described it as follows: 
\begin{quotation}
``One can think of performing a crossing change as grabbing two parallel
strands of a knot with opposite orientation and giving them one full
twist. More generally, one can grab $2k$ parallel strands of $K$
with $k$ of the strands oriented in each direction and give them
one full twist.''
\end{quotation}
Following Livingston, we call such a twist a \emph{generalized crossing
change}. We describe in Section \ref{sub:Knot-surgery} how a crossing
change may be encoded as a $\pm1$-surgery on a nullhomologous unknot
$U\subset S^{3}-K$ bounding a disk $D$ such that $D\cap K=2$ points.
From this perspective, a generalized crossing change is a relaxing
of the previous definition to allow $D\cap K=2k$ points for any $k$,
provided $\text{lk}(K,U)=0$ (see Fig. \ref{fig:Generalized-crossing-change}).
In particular, any knot can be unknotted by a finite sequence of generalized
crossing changes.

\begin{figure}
\def\svgwidth{0.75\columnwidth}
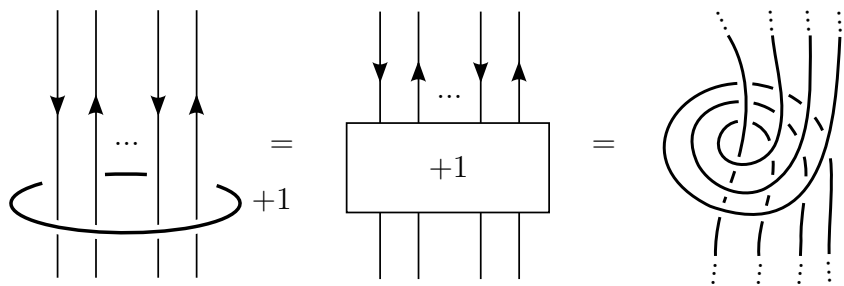

\protect\caption{\label{fig:Generalized-crossing-change}A right-handed, or positive,
generalized crossing change.}
\end{figure}

One may then naturally define the \emph{untwisting number} $tu(K)$
to be the minimum length, taken over all diagrams of $K$, of a sequence
of generalized crossing changes beginning at $K$ and resulting in
the unknot. By $tu_{p}(K)$, we will denote the minimum number of
twists on $2p$ or fewer strands needed to unknot $K$; notice that
$tu_{1}(K)=u(K)$ and that 
\[
tu\leq\dots\leq tu_{p+1}\leq tu_{p}\leq\dots\leq tu_{1}=u.
\]

The \emph{algebraic untwisting number} $tu_{a}(K)$ is the minimum
number of generalized crossing changes, taken over all diagrams of
$K$, needed to transform $K$ into an Alexander polynomial-one knot.
It is clear that $tu_{a}(K)\leq tu(K)$ for all knots $K$.

It is natural to ask how $tu(K)$ and $u(K)$ are related. We show
that these invariants are ``algebraically the same'' in the following
sense: 
\begin{thm}
For any knot $K\subset S^{3}$, $tu_{a}(K)=u_{a}(K)$.
\end{thm}
Therefore, $tu$ and $u$ cannot be distinguished by classical invariants.
By using the Jones polynomial, which is not a classical invariant,
we can show that $tu$ and $u$ are not equal in general:
\begin{thm}
Let $K$ be the image of $V\subset S^{3}$ in the manifold $M\cong S^{3}$
resulting from $+1$-surgery on the unknot $U\subset S^{3}$ shown
in Figure \ref{fig:Counterex}. Then $tu(K)=1$ but $u(K)>1$.
\end{thm}
\begin{figure}
\def\svgwidth{0.75\columnwidth}
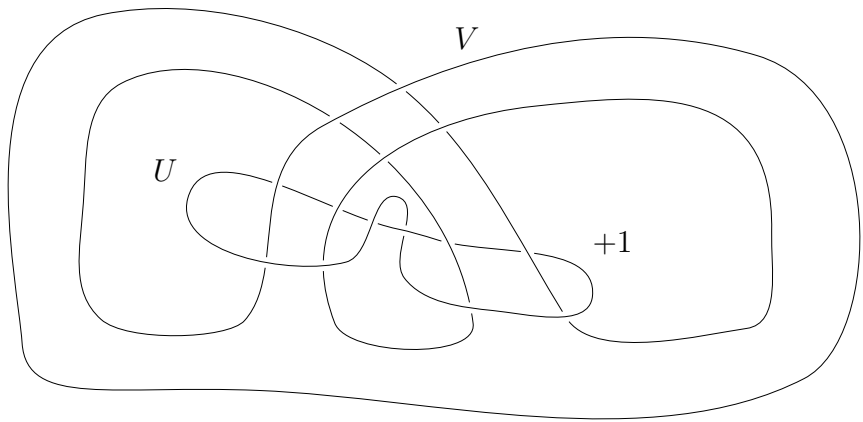

\protect\caption{\label{fig:Counterex}The generalized crossing change for $V\subset S^{3}$
which results in a knot $K\subset S^{3}$ with $tu(K)\protect\neq u(K)$.}
\end{figure}

Furthermore, using the fact that the absolute value of the Ozsv\'{a}th-Szab\'{o} $\tau$ invariant is a lower bound on unknotting number,
we show in Subsection \ref{sub:infinite-gaps-for-gu-1} that the difference $u-tu_{p}$ can be arbitrarily large,
and thus so can the difference $u-tu$. Throughout this paper, $K_{p,q}$
will denote the $(p,q)$-cable of the knot $K$, where $p$ denotes
the longitudinal winding and $q$ the meridional winding.
\begin{thm}
Let $K$ be a knot in $S^{3}$ such that $u(K)=1$. If $\tau(K)>0$
and $p,q>0$, then 
\[
u(K_{p,q})-tu_{p}(K_{p,q})\geq p-1.
\]
In particular, if we take $q=1$, then $tu_{p}(K_{p,q})=1$, while
$u(K_{p,q})\geq p$. 
\end{thm}
It may seem that the above examples are ``cheating'' in some sense,
as in each of them the number of strands of $K$ passing through the
$\pm1$-framed unknot $U$ in the generalized crossing change diagram
is increasing along with $u(K)$. The following theorem shows that
$u(K)$ can be arbitrarily larger than $tu(K)$ even when we restrict
to doing $q$-generalized crossing changes for any fixed integer $q\geq1$.
\begin{thm}
For any knot $K$ with $u(K)=1$ and $\tau(K)>0$, the infinite family
of knots $J_{p}^{q}:=\#^{p}K_{q,1}$ satisfies
\[
u(J_{p}^{q})-tu_{q}(J_{p}^{q})\geq p
\]
 for any integers $p>1,q>0$.
\end{thm}
So far, all of the families of knots we have worked with are quite
complicated, in the sense that they are $(p,q)$-cables for large
$p$ or connected sums of such cables. One may wonder whether it is
possible to find a ``simpler'' knot $K$ for which $tu(K)<u(K)$.
One measure of ``knot simplicity'' is \emph{topological sliceness};
a knot $K$ is \emph{topologically slice} if there exists a locally
flat disk $D\subset B^{4}$ such that $\partial D=K\subset S^{3}=\partial B^{4}$.
\begin{thm}
For any knot $K$ with $\tau(K)>0$, let $D_{+}(K,0)$ denote the
positive-clasped, untwisted Whitehead double of $K$. Then the knots
$S_{p}^{q}:=\#^{p}(D_{+}(K,0))_{q,1}$ are topologically slice and satisfy
\[
u(S_{p}^{q})-tu_{q}(S_{p}^{q})\geq p
\]
for all integers $p>0,q>0$.
\end{thm}
This paper is organized as follows. First, we will review the operations
of Dehn surgery on knots and knot cabling and define the untwisting
number more precisely. Next, we will give some background on the Blanchfield
form which is necessary to prove that $tu_{a}=u_{a}$. Finally, we
will prove that each of the above families of knots give arbitrarily
large gaps between $u$ and $tu$.

\textbf{Convention. }In this paper, all manifolds are assumed to be
compact, orientable, and connected.

\section{Preliminaries}

\subsection{\label{sub:Knot-surgery}Dehn surgery}

In this section, we will describe the operation of Dehn surgery on
knots.
\begin{defn}
\label{def:Knot-surgery}Let $K\subset S^{3}$ be an oriented knot
and $U\subset S^{3}$ be an unknot with $\text{lk}(K,U)=0$. Let $W$
be a closed tubular neighborhood of $U$ in $S^{3}$. Let $\lambda$ be a longitude of $W$, and let $\mu$ be a meridian of $W$ such that $\text{lk}(\mu, \lambda)=1$. The $3$-manifold 
\[
M=(S^{3}-\mathring{W})\bigcup_{h}W,
\]
where $h:\partial W\to\partial W$ is a homeomorphism taking a meridian
of $W$ onto $\pm\mu + \lambda\subset W$, is the \emph{result
of $\pm1$-surgery on $U$}, and $U$ is said to be $\pm1$\emph{-framed.
}In this situation, we define a \emph{generalized crossing change
diagram for $K$} to be a diagram of the link $K\cup U$ with the
number $\pm1$ written next to $U$, indicating that $U$ is $\pm1$-framed.
Figure \ref{fig:Counterex} is an example of a generalized crossing
change diagram for the unknot $V$.
\end{defn}
In the general case, note that the complement of $\mathring{N}\supset U$
in $S^{3}$ is a solid torus, which we may modify with a meridinal
twist. This alters $K$ as follows: if $D$ is a disk bounded by $U$
such that $k$ strands of $K$ pass through $D$ in straight segments,
then each of the $k$ straight pieces is replaced by a helix which
screws through a neighborhood of $D$ in the right-hand sense (see
Fig. \ref{fig:Rolfsen-twist}).

\begin{figure}
\includegraphics[scale=0.3]{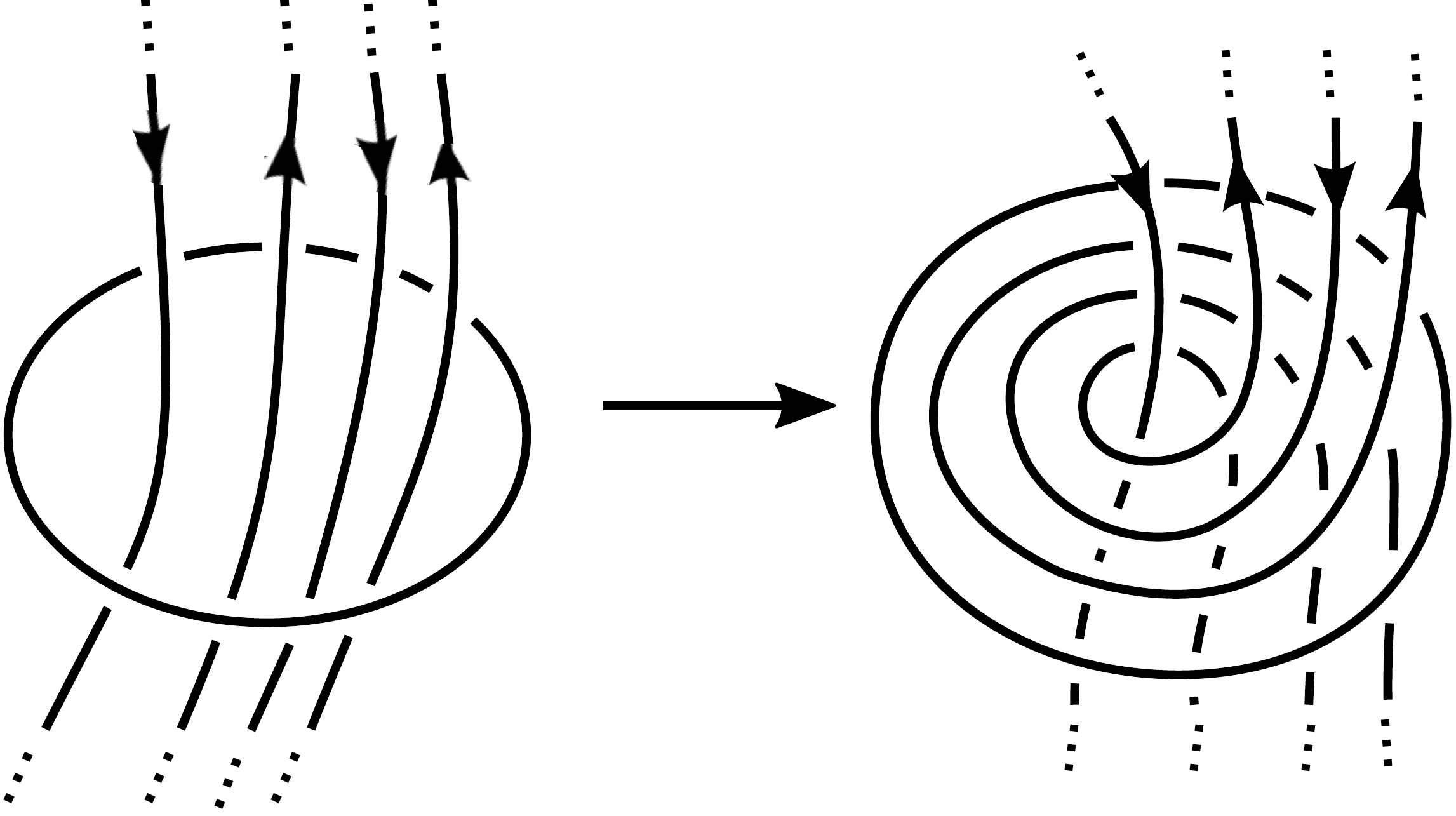}

\protect\caption{\label{fig:Rolfsen-twist}A right-handed twist about an unknotted
component.}
\end{figure}

If $U$ is $-1$-framed, the knot obtained by erasing $U$ and twisting
the strands of $K$ that pass through $U$ as in Figure \ref{fig:Rolfsen-twist}
represents the image of $K$ under the $-1$-surgery on $U$ \cite{rolfsen_knots_1976}.
If instead $U$ has framing $+1$, the knot obtained by erasing $U$
and giving $K$ a left-handed meridinal twist represents the image of $K$ under the $+1$-surgery
on $U$. The process of performing a $\mp$-meridinal twist on the
complement of a $\pm1$-framed unknot $U$, then erasing $U$ from
the resulting diagram, is called a \emph{blow-down on}\textbf{ $U$.
}The inverse process of introducing an unknotted component $U$ to
a surgery diagram consisting of a knot $K$, then performing
a $\pm$-meridinal twist on the complement of $U$ it to link it with
$K$, is known as a \emph{blow-up on}\textbf{ $U$ }and results in
a diagram consisting of $K$ and the $\mp1$-framed unknot $U$, where
$\text{lk}(K,U)=0$.

\begin{figure}
\def\svgwidth{0.5\columnwidth}
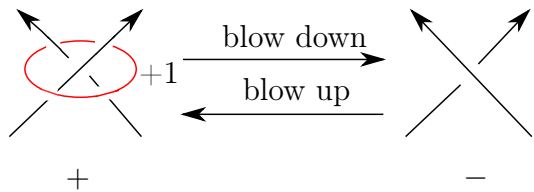

\protect\caption{\label{fig:Crossing-changes-as-blowdowns}Crossing changes as blow-downs
of $\pm1$-framed unknots.}
\end{figure}

Now, it can be easily verified that blowing down the $+1$-framed
unknot on the left side of Figure \ref{fig:Crossing-changes-as-blowdowns}
transforms the crossing labeled $+$ into the crossing labeled $-$.
The inverse process of introducing an unknot to the right side of
Figure \ref{fig:Crossing-changes-as-blowdowns} and performing a $-$-meridinal
twist on its complement yields the positive crossing.

\subsection{Untwisting number}

We define a \textbf{$\pm$}\emph{-generalized crossing change} \emph{on
$K$ }as the process of blowing down the $\pm1$-framed unknot in
a generalized crossing change diagram for $K$. In this situation,
$K$ must pass through $U$ an even number of times, for otherwise
$\text{lk}(K,U)\neq0$. If at most $2p$ strands of $K$ pass through
$U$ in a generalized crossing change diagram, we may call the associated
$\pm$-generalized crossing change a $\pm p$\emph{-generalized crossing
change on $K$.} 

The \emph{result of a $\pm$-generalized crossing change on $K$ }is
defined to be the image of $K$ under the blow-down. The \emph{untwisting
number }$tu(K)$ \emph{of }$K$ is the minimum length of a sequence
of generalized crossing changes on $K$ such that the result of the
sequence is the unknot, where we allow ambient isotopy of the diagram
in between generalized crossing changes. Note that by the reasoning
on page 58 of \cite{adams_knot_1994}, this definition is equivalent
to taking the minimum length, over all diagrams of $K$, of a sequence
of generalized crossing changes beginning with a fixed diagram of
$K$ such that the result of the sequence is the unknot, where we
do not allow ambient isotopy of the diagram in between generalized
crossing changes. 

For $p=1,2,3,\dots$, we define the \textbf{$p$}\emph{-untwisting
number }$tu_{p}(K)$ to be the minimum length of a sequence of $\pm p$-generalized
crossing changes on $K$ resulting in the unknot, where we allow ambient
isotopy of the diagram in between generalized crossing changes.

It follows immediately that we have the chain of inequalities
\begin{equation}
tu(K)\leq\dots\leq tu_{p+1}(K)\leq tu_{p}(K)\leq\dots\leq tu_{2}(K)\leq tu_{1}(K)=u(K).\label{eq:gu_p-chain}
\end{equation}

\subsection{Cabling}

In this section, we define satellite and cable knots.
\begin{defn}
A closed subset $X$ of a solid torus $V\cong S^{1}\times D^{2}$
is called \emph{geometrically essential} in $V$ if $X$ intersects
every PL meridinal disk in $V$.

Let $P\subset V\subset S^{3}$ be a knot which is geometrically essential
in an unknotted solid torus $V$. Let $C\subset S^{3}$ be another
knot and let $V_{1}$ be a tubular neighborhood of $C$ in $S^{3}$.
Let $h\colon V\to V_{1}$ be a homeomorphism and let $K$ be $h(P)$.
Then $P$ is called the \emph{pattern} for the knot $K$, $C$ is
the \emph{companion}\textbf{ }of $K$, and $K$ is called a \emph{satellite
of $C$ with pattern $P$}, or just a \emph{satellite knot} for short.

If the homeomorphism $h$ takes the preferred longitude and meridian
of $V$, respectively, to the preferred longitude and meridian of
$V_{1}$, then $h$ is said to be \emph{faithful}. If $P$ is the
$(p,q)$-torus knot just under $\partial V$ and $h$ is faithful,
then $K$ is called the $(p,q)$\emph{-cable based on }$C$, denoted
$C_{p,q}$, or simply a \emph{cable knot.}
\end{defn}
Throughout this paper, we will denote the $(p,q)$-torus knot by $U_{p,q}$
since it is the $(p,q)$-cable of the unknot $U$.

\subsection{The Blanchfield form}

Let $K\subset S^{3}$ be a knot. By $\Lambda$ we shall denote the
ring $\mathbb{Z}[t^{\pm1}]$, and by $\Omega$ we will denote the
field $\mathbb{Q}(t)$.

\subsubsection{Twisted homology, cohomology groups, and Poincar\'{e} duality}

Following \cite{borodzik_algebraic_2014}, let $X$ be a manifold
with infinite cyclic first homology, and fix a choice of isomorphism
of $H_{1}(X)$ with the infinite cyclic group generated by the indeterminate
$t$. Let $\pi\colon\widetilde{X}\to X$ be the infinite cyclic cover
of $X$. Given a submanifold $Y$ of $X$, let $\widetilde{Y}=\pi^{-1}(Y)$.
Since $\mathbb{Z}$ is the deck transformation group of $\widetilde{X}$,
$\Lambda$ acts on the relative chain group $C_{\ast}(\widetilde{X},\widetilde{Y};\mathbb{Z})$.
If $N$ is any $\Lambda$-module, we may define
\[
H^{\ast}(X,Y;N):=H_{\ast}(\text{Hom}_{\Lambda}(C_{\ast}(\widetilde{X},\widetilde{Y};\mathbb{Z}),N))
\]
and 
\[
H_{\ast}(X,Y;N):=H_{\ast}(\overline{C_{\ast}(\widetilde{X},\widetilde{Y};\mathbb{Z}})\otimes_{\Lambda}N).
\]
Here, if $H$ is any $\Lambda$-module, $\overline{H}$ denotes
the module with the involuted $\Lambda$-structure: multiplication by $p(t)\in\Lambda$ in $\overline{H}$
is the same as multiplication by $p(t^{-1})$ in $H$. When $Y=\emptyset$,
we just write $H_{\ast}(X;N)$ or $H^{\ast}(X;N)$.

Since $\Omega:=\mathbb{Q}(t)$ is flat over $\Lambda$, we have isomorphisms $H_{\ast}(X,Y;\Omega)\cong H_{\ast}(X,Y;\Lambda)\otimes_{\Lambda}\Omega$
and $H^{\ast}(X,Y;\Omega)\cong H^{\ast}(X,Y;\Lambda)\otimes_{\Lambda}\Omega$. If $X$ is an $n$-manifold, and $N$ is a $\Lambda$-module, Poincar\'{e} duality gives $\Lambda$-module isomorphisms
\[
H_{i}(X,\partial X;N)\cong H^{n-i}(X;N).
\]

\subsubsection{The Blanchfield form}

As above, let $\Lambda=\mathbb{Z}[t,t^{-1}]$ and $\Omega=\mathbb{Q}(t)$.
Let $A$ be an $n\times n$ invertible hermitian matrix with entries in $\Lambda$. 
We define $\lambda(A)$ to be the pairing
\[
\lambda(A)\colon\Lambda^{n}/A\Lambda^{n}\times\Lambda^{n}/A\Lambda^{n}\to\Omega/\Lambda
\]
sending the pair of column vectors $(a,b)$ to $\bar{a}^{t}A^{-1}b$.
Note that $\lambda(A)$ is a nonsingular, hermitian pairing.

Let $X(K)=S^{3}-N(K)$ denote the exterior of $K$. Consider the following
sequence of maps:
\begin{eqnarray*}
\Phi\colon H_{1}(X(K);\Lambda) & \xrightarrow{\pi_{\ast}} & H_{1}(X(K),\partial X(K);\Lambda)\\
 & \xrightarrow{PD} & H^{2}(X(K);\Lambda)\xleftarrow{\delta}H^{1}(X(K);\Omega/\Lambda)\\
 & \xrightarrow{ev} & \overline{\text{Hom}_{\lambda}(H_{1}(X(K);\Lambda),\Omega/\Lambda)}.
\end{eqnarray*}
Here $\pi_{\ast}$ is induced by the quotient map $C(X) \to C(X)/C(\partial X)$, $PD$ is the Poincar\'{e} duality
map, $\delta$ is from the long exact sequence in cohomology obtained
from the coefficients $0\to\Lambda\to\Omega\to\Omega/\Lambda\to0$,
and $ev$ is the Kronecker evaluation map. It is well-known (see \cite[Section 2]{hillman_algebraic_2012} for details) that $\pi_{\ast}$
and $\delta$ are isomorphisms, $PD$ is the Poincar\'{e} duality
isomorphism, and $ev$ is also an isomorphism by the universal coefficient
spectral sequence (see \cite[Theorem 2.3]{levine_knot_1977} for details
on the universal coefficient spectral sequence). Thus, the above maps
define a nonsingular pairing
\begin{eqnarray*}
\lambda(K)\colon H_{1}(X(K);\Lambda)\times H_{1}(X(K);\Lambda) & \to & \Omega/\Lambda\\
(a,b) & \mapsto & \Phi(a)(b),
\end{eqnarray*}
called the \emph{Blanchfield pairing of $K$}. It is well-known that this pairing is hermitian.

Now, let $V$ be any $2k\times2k$ matrix which is $S$-equivalent
to a Seifert matrix for $K$. Recall that $V-V^{T}$ is antisymmetric
with determinant $\pm 1$. It is well-known that, perhaps after
replacing $V$ by $PVP^{T}$ for some $P \in GL_{2k}(\mathbb{Z})$, 
\begin{equation}
V-V^{T}=\begin{pmatrix}0 & -I_{k}\\
I_{k} & 0
\end{pmatrix}, \label{eq:V-VT-form}
\end{equation}
where $I_{k}$ denotes the $k\times k$ identity matrix. We define
$A_{K}(t)$ to be the matrix
\[
\begin{pmatrix}(1-t^{-1})^{-1}I_{k} & 0\\
0 & I_{k}
\end{pmatrix}V\begin{pmatrix}I_{k} & 0\\
0 & (1-t)I_{k}
\end{pmatrix}+\begin{pmatrix}I_{k} & 0\\
0 & (1-t^{-1})I_{k}
\end{pmatrix}V^{T}\begin{pmatrix}(1-t)^{-1}I_{k} & 0\\
0 & I_{k}
\end{pmatrix}.
\]
Using (\ref{eq:V-VT-form}), we can write 
\[
V=\begin{pmatrix}B & C+I\\
C^{T} & D
\end{pmatrix}.
\]
One may then compute, as in the proof of \cite[Lemma 2.2]{borodzik_unknotting_2015}, that
\[
A_{K}(1)=\begin{pmatrix}B & -I_{k}\\
-I_{k} & 0
\end{pmatrix}.
\]
Thus, the matrix $A_{K}(t)$ is a Hermitian matrix defined over $\Lambda$, and $\det(A_{K}(1))=(-1)^{k}$.
\begin{prop}
\cite[Proposition 2.1]{borodzik_unknotting_2015} Let $K$ be a knot
and $A_{K}(t)$ be as above. Then $\lambda(A_{K}(t))$ is isometric
as a sesquilinear form to $\lambda(K)$.
\end{prop}

\subsection{The twisted intersection pairing}

Let $W$ be a topological $4$-manifold with boundary $M$ such
that $\pi_{1}(W)=\mathbb{Z}$. Consider the maps
\[
H_{2}(W;\Lambda)\xrightarrow{\pi_{\ast}}H_{2}(W,M;\Lambda)\xrightarrow{PD}H^{2}(W;\Lambda)\xrightarrow{ev}\overline{\text{Hom}_{\Lambda}(H_{2}(W;\Lambda),\Lambda)},
\]
where the first map is induced by the quotient, the second map is Poincar\'{e}
duality, and the third map is the Kronecker evaluation map. The second and third maps are obviously isomorphisms, and the first map is an isomorphism by the long exact sequence of the pair $(W,M)$. Hence this composition
defines a pairing
\[
H_{2}(W;\Lambda)\times H_{2}(W;\Lambda)\to\Lambda
\]
which we call the \emph{twisted intersection pairing on $W$}.

\section{Algebraic untwisting number equals algebraic unknotting number}

Our proof that $tu_{a}(K)=u_{a}(K)$ generalizes the work of Borodzik
and Friedl in \cite{borodzik_algebraic_2014,borodzik_unknotting_2015}.
Following  \cite{borodzik_algebraic_2014}, define a knot invariant $n(K)$ 
to be the minimum size of a square Hermitian matrix $A(t)$ over $\mathbb{Z}[t^{\pm 1}]$ 
such that $\lambda(A)$ is isometric to $\lambda(K)$ and $A(1)$ is congruent over $\mathbb{Z}$ to a diagonal matrix
which has only $\pm 1$ entries.
Borodzik and Friedl showed that
$u_{a}(K)=n(K)$. Since $tu_{a}(K)\leq u_{a}(K)$,
it is obvious that $tu_{a}(K)\leq n(K)$ as well. After stating Borodzik
and Friedl's results, we will show that $n(K)\leq tu_{a}(K)$, hence
$tu_{a}(K)=n(K)=u_{a}(K)$ for all knots $K$. In fact, we will show
something stronger. 
\begin{thm}
\label{thm:Signed-algebraic-equality}Let $K\subset S^{3}$ be a knot.
For every algebraic unknotting sequence for $K$ with $u_{+}$ positive
crossing changes and $u_{-}$ negative crossing changes, there exists
an algebraic untwisting sequence for $K$ with $u_{+}$ positive generalized
crossing changes and $u_{-}$ negative generalized crossing changes.
In particular, $u_{a}(K)=tu_{a}(K)$.
\end{thm}
In order to prove Theorem \ref{thm:Signed-algebraic-equality}, we
must first recall some notation and results used by Borodzik and Friedl in
\cite{borodzik_unknotting_2015}. The main theorem of \cite{borodzik_unknotting_2015}
implies that $n(K)\leq u_{a}(K)$:
\begin{thm}
\cite[Theorem 1.1]{borodzik_unknotting_2015} \label{thm:(BF-Thm-1.1)}Let $K$ be a knot which
can be changed into an Alexander polynomial-one knot by a sequence
of $u_{+}$ positive crossing changes and $u_{-}$ negative crossing
changes. Then there exists a hermitian matrix $A(t)$ of size $u_{+}+u_{-}$
over $\Lambda$ such that
\begin{enumerate}
\item $\lambda(A(t))$ is isometric to $\lambda(K)$;
\item $A(1)$ is a diagonal matrix such that $u_{+}$ diagonal entries are
equal to $-1$ and $u_{-}$ diagonal entries are equal to $1$.
\end{enumerate}
In particular, $n(K)\leq u_{a}(K)$.
\end{thm}
We need one definition:
\begin{defn}
Let $K$ be a knot and $M(K)$ the result of $0$-surgery on
$K$. A $4$-manifold $W$ \emph{tamely cobounds}\textbf{
$M(K)$ }if:
\begin{enumerate}
\item $\partial W=M(K)$;
\item the inclusion induced map $H_{1}(M(K);\mathbb{Z})\to H_{1}(W;\mathbb{Z})$
is an isomorphism;
\item $\pi_{1}(W)=\mathbb{Z}$.
\end{enumerate}
If, in addition, the intersection form on $H_{2}(W;\mathbb{Z})$ is
diagonalizable, we say that $W$ \emph{strictly cobounds}\textbf{
}$M(K)$.
\end{defn}
The following theorem of Borodzik-Friedl is also needed:
\begin{thm}
\label{thm:(Borodzik-Friedl-Theorem-2.6)}\cite[Theorem 2.6]{borodzik_unknotting_2015}
Let $K$ be a knot and let $W$ be a topological $4$-manifold which
tamely cobounds $M(K)$. Then $H_{2}(W;\Lambda)$ is free of rank
$b_{2}(W)$. Moreover, if $B$ is an integral matrix representing
the ordinary intersection pairing of $W$, then there exists a basis
$\mathcal{B}$ for $H_{2}(W;\Lambda)$ such that the matrix $A(t)$
representing the twisted intersection pairing with respect to $\mathcal{B}$
satisfies
\begin{enumerate}
\item $\lambda(A(t))$ is isometric to $\lambda(K)$;
\item $A(1)=B$.
\end{enumerate}
\end{thm}
We generalize Theorem \ref{thm:(BF-Thm-1.1)} as follows:
\begin{thm}
\label{thm:(generalized-BF-2.6)} Let $K$ be a knot which can be changed into an Alexander polynomial-one
knot by a sequence of $u_{+}$ positive and $u_{-}$ negative generalized
crossing changes. Then there exists a hermitian matrix of size $u_{+}+u_{-}$
over $\Lambda$ with the following two properties:
\begin{enumerate}
\item $\lambda(A(t))$ is isometric to $\lambda(K)$;
\item $A(1)$ is a diagonal matrix such that $u_{+}$ diagonal entries are
equal to $-1$ and $u_{-}$ diagonal entries are equal to $1$.
\end{enumerate}
In particular, $n(K)\leq tu_{a}(K)$.
\end{thm}

The proof of Theorem \ref{thm:(generalized-BF-2.6)} is quite similar to the proof of Theorem \ref{thm:(BF-Thm-1.1)}. By Theorem \ref{thm:(Borodzik-Friedl-Theorem-2.6)}, in order to prove Theorem \ref{thm:(generalized-BF-2.6)}, we only need to show the following proposition.
\begin{prop}
\label{prop:Borodzik-Friedl-generalization}Let $K$ be a knot such
that $u_{+}$ positive generalized crossing changes and $u_{-}$ negative
generalized crossing changes turn $K$ into an Alexander polynomial-one
knot. Then there exists an oriented topological $4$-manifold $W$
which strictly cobounds $M(K)$. Moreover, the intersection pairing
on $H_{2}(W;\mathbb{Z})$ is represented by a diagonal matrix of size
$u_{+}+u_{-}$ such that $u_{+}$ entries are equal to $-1$ and $u_{-}$
entries are equal to $+1$.\end{prop}

\begin{proof}
Let $K$ be a knot such that $u_{+}$ positive generalized crossing
changes and $u_{-}$ negative generalized crossing changes turn $K$
into an Alexander polynomial-one knot $J$. We write $s=u_{+}+u_{-}$
and $n_{i}=-1$ for $i=1,\dots,u_{+}$ and $n_{i}=1$ for $i=u_{+}+1,\dots,u_{+}+u_{-}$.
Then there exist simple closed curves $c_{1},\dots,c_{s}$ in $S^{3}-N(J)$
such that
\begin{enumerate}
\item $c_{1}\cup\dots\cup c_{s}$ is the unlink in $S^{3}$;
\item the linking numbers $\text{lk}(c_{i},K)$ are zero for all $i$;
\item the image of $J$ under the $n_{i}$-surgeries is the knot $K$.
\end{enumerate}
Note that the curves $c_{1},\dots,c_{s}$ lie in $S^{3}-N(J)$, hence
we can view them as lying in $M(J)$. The manifold $M(K)$ is then
the result of $n_{i}$-surgery on all the $c_{i}\subset M(J)$, where
$i=1,\dots,s$.

Since $J$ is a knot with trivial Alexander polynomial, by Freedman's theorem \cite{freedman_topology_2014} $J$ is
topologically slice and there exists a locally flat slice disk
$D\subset B^{4}$ for $J$ such that $\pi_{1}(B^{4}-D)=\mathbb{Z}$.
Let $X:=B^{4}-N(D)$. Then $X$ is an oriented topological
$4$-manifold such that
\begin{enumerate}
\item $\partial X\cong M(J)$ as oriented manifolds;
\item $\pi_{1}(X)\cong\mathbb{Z}$;
\item the inclusion induced map $H_{1}(M(J);\mathbb{Z})\to H_{1}(X;\mathbb{Z})$ is an isomorphism;
\item $H_{2}(X;\mathbb{Z})=0$.
\end{enumerate}
Let $W$ be the $4$-manifold which is obtained by adding $2$-handles
along $c_{1},\dots,c_{s}\subset M(J)$ with framings $n_{1},\dots,n_{s}$
to $X$. Then $\partial W\cong M(K)$ as oriented manifolds.
From now on, we write $M:=M(K)$. Since the curves $c_{1},\dots,c_{s}$
are nullhomologous, the map $H_{1}(M;\mathbb{Z})\to H_{1}(W;\mathbb{Z})$
is an isomorphism and $\pi_{1}(W)\cong\mathbb{Z}$. It thus remains
to prove the following lemma: 
\begin{lem}
\label{lem:(claim)}The ordinary intersection pairing on $W$ is represented by a diagonal
matrix of size $s=u_{+}+u_{-}$ with $u_{+}$ diagonal entries equal
to $-1$ and $u_{-}$ diagonal entries equal to $1$.
\end{lem}
Recall that the curves $c_{1},\dots,c_{s}$ form the unlink in $S^{3}$
and that the linking numbers $\text{lk}(c_{i},J)$ are zero. Therefore,
the curves $c_{1},\dots,c_{s}$ are also nullhomologous in $M(J)$.
Thus we can now find disjoint surfaces $F_{1},\dots,F_{s}$
in $M(J)\times[0,1]$ such that $\partial F_{i}=c_{i}\times\{1\}$. By
adding the cores of the $2$-handles attached to the $c_{i}$, we
obtain closed surfaces $C_{1},\dots,C_{s}$ in $W$. It is clear that
$C_{i}\cdot C_{j}=0$ for $i\neq j$ and $C_{i}\cdot C_{i}=n_{i}$.

We argue using Mayer-Vietoris that the surfaces $C_{1},\dots,C_{s}$
present a basis for $H_{2}(W;\mathbb{Z})$. Write $W:=X\cup H$ where
$H\cong\sqcup_{i=1}^{s}(B^{2}\times B^{2})$ is the set of $2$-handles
attached to $c_{1},\dots,c_{s}$. Then write $Y:=X\cap H$, so that
\[
Y=\sqcup_{i=1}^{s}N(c_{i})\cong\sqcup_{i=1}^{s}(S^{1}\times D^{2}).
\]
We have the Mayer-Vietoris sequence
\[
\dots\to H_{2}(X)\oplus H_{2}(H)\xrightarrow{\psi_{\ast}}H_{2}(W)\xrightarrow{\partial_{\ast}}H_{1}(Y)\xrightarrow{\phi_{\ast}}H_{1}(X)\oplus H_{1}(H)\xrightarrow{\psi_{\ast}}H_{1}(W)\to0.
\]

Now, since $H_{1}(Y)$ is generated by all the $S^{1}$-factors, or
the longitudes $c_{1},\dots,c_{s}$, and $H_{1}(H)=H_{2}(H)=H_{2}(X)=0$,
the sequence becomes 
\[
0\to H_{2}(W)\xrightarrow{\partial_{\ast}}\langle c_{1},\dots,c_{s}\rangle\xrightarrow{i_{\ast}}H_{1}(X)\xrightarrow{\psi_{\ast}}H_{1}(W)\to0.
\]

From e.g. \cite[Lemma 8.12]{livingston_knot_1993}, we have
\begin{lem}
Suppose that for some knot $K$ in $S^{3}$, there is a locally flat
surface $F$ in $B^{4}$ with $F\cap S^{3}=\partial F\cap S^{3}=K$.
Then the inclusion map induces an isomorphism $H_{1}(S^{3}-K)\to H_{1}(B^{4}-F)\cong\mathbb{Z}$.
\end{lem}
In our case, the inclusion $S^{3}-K\hookrightarrow X$ induces an
isomorphism $H_{1}(S^{3}-K)\to H_{1}(X)$. Since $i_{\ast}$ is induced
by inclusion and the longitudes $c_{1},\dots,c_{s}$ are nullhomologous
in $S^{3}-K$, hence in $X$, $i_{\ast}$ must be the zero map. Hence
$\partial_{\ast}$ is an isomorphism $H_{2}(W)\cong H_{1}(Y)$, and
$H_{2}(W)=\langle C_{1},\dots,C_{s}\rangle$.

In particular, the intersection matrix on $W$ with respect to this
basis is given by $(C_{i}\cdot C_{j})$, i.e. it is a diagonal matrix
such that $u_{+}$ diagonal entries are equal to $-1$ and $u_{-}$
diagonal entries are equal to $+1$. This concludes the proof of Lemma \ref{lem:(claim)}. Proposition \ref{prop:Borodzik-Friedl-generalization} follows. Together with Theorem \ref{thm:(Borodzik-Friedl-Theorem-2.6)}, this completes the proof of Theorem \ref{thm:(generalized-BF-2.6)}.
\end{proof}
We have shown that, for every untwisting sequence for $K$ with $u_{+}$
positive generalized crossing changes and $u_{-}$ negative generalized
crossing changes, there exists a hermitian matrix $A(t)$ of size
$u_{+}+u_{-}$ such that $\lambda(A(t))$ is isometric to $\lambda(K)$ and $A(1)$
is diagonal with $u_{+}$ entries equal to $-1$ and $u_{-}$ entries
equal to $1$. Borodzik and Friedl \cite{borodzik_algebraic_2014}
have already shown that, for every hermitian matrix $A(t)$ representing
$\lambda(K)$ such that $A(1)$ is diagonal with $u_{+}$ $-1$'s
and $u_{-}$ $+1$'s, there exists an algebraic unknotting sequence
for $K$ consisting of $u_{+}$ positive and $u_{-}$ negative crossing
changes. Theorem \ref{thm:Signed-algebraic-equality} follows.

\section{Untwisting Number Does Not Equal Unknotting Number}

Although the algebraic versions of $tu$ and $u$ are equal, $tu\neq u$
in general. We use a result of Miyazawa \cite{miyazawa_jones_1998}
to give our first example of a knot $K$ with $tu(K)=1$ but $u(K)>1$.
\begin{thm}
Let $K$ be the knot resulting from blowing down the $(+1)$-framed
unknot $U\subset S^{3}\setminus V$ in Figure \ref{fig:Counterex}.
Then $tu(K)=1$ but $u(K)>1$.
\end{thm}

From this point forward, we will denote the signature of any knot
$K$ by $\sigma(K)$. In order to analyze the unknotting number of
$K$, we will use the following theorem:
\begin{thm}
\label{thm:(Miyazawa-1998)}\cite{miyazawa_jones_1998} If $u(K)=1$
and $\sigma(K)=\pm2$, then
\[
V_{K}^{(1)}(-1)\equiv24a_{4}(K)-\frac{\sigma(K)}{8}(\det K+1)(\det K+5)\text{ (mod 48)}
\]
where $V_{K}^{(1)}$ denotes the first derivative of the Jones polynomial
of $K$ and $a_{4}$ is the coefficient of $z^{4}$ in the Conway
polynomial $\nabla_{K}(z)=\sum_{n=0}^{\infty}a_{2n}(K)z^{2n}$.
\end{thm}
We compute using the Mathematica package KnotTheory \cite{mathematica_????}
that $\sigma(K)=2$, hence Theorem \ref{thm:(Miyazawa-1998)} applies.
We also compute using the KnotTheory package that the Jones polynomial
$V_{K}(q)$ for our knot $K$ is 
\[
V_{K}(q)=q-q^{2}+2q^{3}-q^{4}+q^{6}-q^{7}+q^{8}-q^{9}-q^{12}+q^{13},
\]
hence $V_{K}^{(1)}(-1)=8$. The Conway polynomial of $K$ is computed
to be 
\[
\nabla_{K}(z)=\sum_{n=0}^{\infty}a_{2n}(K)z^{2n}=1+z^{2}
\]
(hence $a_{4}=0$), and the determinant of $K$ is $3$. In our case,
the right-hand side of the congruence in Theorem \ref{thm:(Miyazawa-1998)}
becomes
\[
0-\frac{1}{4}(4)(8)=-8
\]
and $8\not\equiv-8$ (mod $48$). Hence $K$ cannot have unknotting
number one, although it was constructed to have untwisting number
one. Note that this also shows Miyazawa's Jones polynomial criterion
does not extend to untwisting number-one knots.

\section{Arbitrarily large gaps between unknotting and untwisting numbers}

\subsection{\label{sub:infinite-gaps-for-gu-1}Arbitrarily large gaps between
\texorpdfstring{$u$ and $tu_{p}$}{unknotting number and p-untwisting number}}

Now that we have shown that there exists a knot $K$ with $tu(K)<u(K)$,
it is natural to ask how large the difference $u(K)-tu(K)$ can be.
Recall that the $(p,q)$-cable of a knot $K$ is denoted $K_{p,q}$;
we denote the $(p,q)$-torus knot as $U_{p,q}$, the $(p,q)$-cable
of the unknot. The knots we will be working with are $(p,q)$-cables
of knots $K$ with $u(K)=1$ and $\tau(K)>0$, where $p,q>0$.

In order to get a lower bound on $u(K_{p,q})$ for such knots, we
compute $\tau(K_{p,q})$ for all $p,q$. For cables of alternating
(or more generally, ``homologically thin'') knots such as the trefoil,
Petkova \cite{petkova_cables_2013-1} gives a formula for computing
$\tau$. However, since we will later compute $\tau$ for cables of
non-alternating knots, we use a more general method of computing $\tau(K_{p,q})$
using the $\epsilon$-invariant $\epsilon(K)\in\{-1,0,1\}$ introduced
by Hom in \cite{hom_bordered_2014}:
\begin{thm}
\label{thm:(Hom-2014)-Let} \cite{hom_bordered_2014} Let $K\subset S^{3}$. 
\begin{enumerate}
\item If $\epsilon(K)=1$, then $\tau(K_{p,q})=p\tau(K)+\frac{(p-1)(q-1)}{2}$.
\item If $\epsilon(K)=-1$, then $\tau(K_{p,q})=p\tau(K)+\frac{(p-1)(q+1)}{2}$.
\item If $\epsilon(K)=0$, then $\tau(K)=0$ and $\tau(K_{p,q})=\tau(U_{p,q})=\begin{cases}
\frac{(p-1)(q+1)}{2}, & q<0\\
\frac{(p-1)(q-1)}{2}, & q>0.
\end{cases}$
\end{enumerate}
\end{thm}
We note the following property of $\tau$:
\begin{thm}
\cite{ozsvath_knot_2003-1} For the $(p,q)$-torus knot $U_{p,q}$
with $p,q>0$, $\tau$ equals the $3$-sphere genus of $U_{p,q}$,
denoted $g(U_{p,q})$: 
\[
\tau(U_{p,q})=g(U_{p,q})=\frac{(p-1)(q-1)}{2}.
\]

\end{thm}
We also need the following proposition of Hom:
\begin{prop}
\label{prop:(Hom-2014)-Let} \cite{hom_bordered_2014} Let $K\subset S^{3}$
be a knot. If $|\tau(K)|=g(K)$, then $\epsilon(K)=\text{sgn }\tau(K)$. \end{prop}
\begin{thm}
Let $K$ be a knot in $S^{3}$ with unknotting number one. If $\tau(K)>0$
and $p,q>0$, then 
\[
u(K_{p,q})-tu_{p}(K_{p,q})\geq p-1.
\]
In particular, $tu_{p}(K_{p,1})=1$, while $u(K_{p,1})\geq p$.\end{thm}
\begin{proof}
Let $V$ be the unknot that results from performing the unknotting
crossing change on $K$. Consider a generalized crossing change diagram
for $V$ together with the $\pm1$-framed surgery curve $U$ that
transforms $V$ back into $K$. Then take the $(p,q)$-cable $V_{p,q}$
of $V$ in this diagram, leaving $U$ alone. The resulting $V_{p,q}$
is the $(p,q)$-torus knot before performing the $\pm1$-surgery,
but the image of $V$ under $\pm1$-surgery on $U$ is $K$, hence
the image of $V_{p,q}$ under the $\pm1$-surgery on $U$ is $K_{p,q}$.
Therefore, blowing down the surgery curve $U$ (through which $V_{p,q}$ passes
$2p$ times) results in a diagram for $K_{p,q}$ in $S^{3}$. Since
$K_{p,q}$ and $V_{p,q}$ differ by a single twist,
\[
tu_{p}(K_{p,q})\leq tu_{p}(V_{p,q})+1.
\]
Since 
\[
tu_{p}(V_{p,q})\leq u(V_{p,q})=\frac{(p-1)(q-1)}{2},
\]
we get that
\[
tu_{p}(K_{p,q})\leq\frac{(p-1)(q-1)}{2}+1.
\]
In particular, this inequality shows that $tu_{p}(K_{p,1})=1$. If
$\tau(K)>0$, then necessarily $\epsilon(K)\neq0$ by $(3)$ of Theorem
\ref{thm:(Hom-2014)-Let}, so that $\epsilon(K)=\pm1$. In this case,
\[
\tau(K_{p,q})=p\tau(K)+\frac{(p-1)(q\mp1)}{2},
\]
and thus
\[
u(K_{p,q})\geq|\tau(K_{p,q})|=p\tau(K)+\frac{(p-1)(q\mp1)}{2}\geq p+\frac{(p-1)(q\mp1)}{2}.
\]
When $q=1$, we get that $u(K_{p,1})\geq p$. Combining our estimates,
\begin{eqnarray*}
u(K_{p,q})-tu_{p}(K_{p,q}) & \geq & \Big(p+\frac{(p-1)(q\mp1)}{2}\Big)-\Big(1+\frac{(p-1)(q-1)}{2}\Big)\\
 & \geq & \Big(p+\frac{(p-1)(q-1)}{2}\Big)-\Big(1+\frac{(p-1)(q-1)}{2}\Big)\\
 & \geq & p-1,
\end{eqnarray*}
as desired. 
\end{proof}

\subsection{Arbitrarily large gaps between \texorpdfstring{$u$ and $tu_{q}$}{unknotting number and q-untwisting number}}

The above examples $\{K_{p,1}\}$ show that, for every $p$, there
exists a knot $K_{p,1}$ with $u(K_{p,1})\geq p$, even though $tu_{p}(K_{p,1})=1$.
However, in order to untwist any such $K_{p,1}$, we must twist at
least $2p$ strands at once. A natural follow-up question is whether
there exists a knot $K$ with $u(K)\geq p$ that can be untwisted
by a single $\pm q$-generalized crossing change, where $q<p$. More
generally, we may ask whether, for any fixed $q$, there is a family
of knots which give us arbitrarily large gaps between $u$ and $tu_{q}$.
We answer this question in the affirmative.
\begin{thm}
Let $K$ be a knot with $u(K)=1$ and $\tau(K)>0$, and let $J_{p}^{q}:=\#^{p}K_{q,1}$.
For any $p>0$ and $q>1$, $tu_{q}(J_{p}^{q})\leq p$, and $u(J_{p}^{q})-tu_{q}(J_{p}^{q})\geq p$.\end{thm}
\begin{proof}
First, we note that for any knot $K$, $J_{p}^{q}=\#^{p}K_{q,1}$
can be unknotted by performing $p$ generalized crossing changes on
at most $2q$ strands each, one generalized crossing change to unknot
each copy of $K_{q,1}$. Therefore, $tu_{q}(J_{p}^{q})\leq p$. Since
$\tau$ is additive under connected sum,
\[
\tau(J_{p}^{q})=p\cdot\tau(K_{q,1})\geq pq
\]
and hence $u(J_{p}^{q})\geq pq$ for all $p$. Therefore, 
\[
u(J_{p}^{q})-tu_{q}(J_{p}^{q})\geq pq-p=p(q-1)\geq p,
\]
as desired.\end{proof}
\begin{note*}
In the case where $K$ has $\sigma(K)=\pm2$, e.g. when $K$ is a
right-handed trefoil knot, we can do better by computing $tu_{q}$
precisely. We use the fact that $|\sigma(K)|/2$ is a lower bound
for $tu_{q}(K)$ for any $q$. First, recall that the \emph{Tristram-Levine
signature function} of a knot $K$, $\sigma_{\omega}(K)$, is equal
to the signature of the matrix $(1-\omega)V+(1-\overline{\omega})V^{T}$,
where $\omega\in\mathbb{C}$ has norm $1$ and $V$ is a Seifert matrix
for $K$. Note that 
\[
\sigma_{-1}(K)=\sigma(2(V+V^{T}))=\sigma(V+V^{T})=\sigma(K).
\]
We use Litherland's \cite{litherland_signatures_1979} formula for
Tristram-Levine signatures of cable knots to compute that
\[
\sigma_{-1}(K_{p,q})=\sigma_{(-1)^{p}}(K)+\sigma_{-1}(U_{p,q})
\]
and, since $\sigma_{1}\equiv0$, while $\sigma_{-1}=\sigma$,
\[
\sigma(K_{q,1})=\begin{cases}
\sigma(K)+\sigma(U_{q,1}), & q\text{ odd}\\
\sigma(U_{q,1}), & q\text{ even}
\end{cases}=\begin{cases}
\sigma(K), & q\text{ odd}\\
0, & q\text{ even}
\end{cases}
\]
since the $(q,1)$-torus knot is the unknot for any $q$. Now, since
the knot signature is additive over connected sum, 
\[
\sigma(J_{p}^{q})=p\sigma(K_{q,1})=\begin{cases}
\sigma(K)\cdot p, & q\text{ odd}\\
0, & q\text{ even}
\end{cases}=\begin{cases}
\pm2p, & q\text{ odd}\\
0, & q\text{ even}
\end{cases}
\]
and therefore, when $p$ is odd,
\[
tu_{q}(J_{p}^{q})\geq\frac{|\sigma(\kappa_{p}^{q})|}{2}=p.
\]
Since we already know $tu_{q}(J_{p}^{q})\leq p$, in fact we must
have $tu_{q}(J_{p}^{q})=p$ for odd $p\geq1$. 
\end{note*}

\subsection{Arbitrarily large gaps between \texorpdfstring{$u$ and $tu_{q}$}{unknotting number and q-untwisting number} for topologically
slice knots}

Consider the diagram of an unknot $U(K)$ in Figure \ref{fig:L},
where $K$ is any knot with $\tau(K)>0$. Let $p\geq2$ be an integer.

\begin{figure}
\def\svgwidth{0.75\columnwidth}
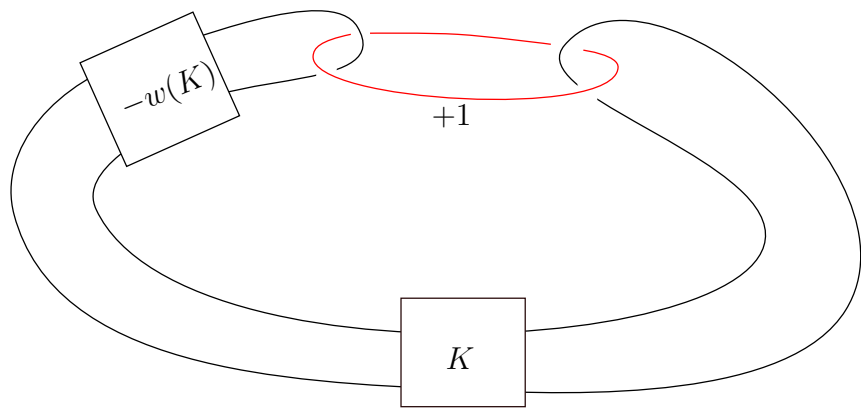

\protect\caption{\label{fig:L}The knot $U(K)$ (an unknot), together with a $+1$-surgery
curve.}
\end{figure}

We take the $(q,1)$-cable of $U(K)$, which is still an unknot. Then,
we perform a $(-1)$-twist on the $(+1)$-framed unknot, obtaining
a knot $S^{q}$. Clearly, $tu_{q}(S^{q})=1.$ 

Furthermore, $S^{q}$ is the $(q,1)$-cable of the knot $D_{+}(K,0)$,
the untwisted Whitehead double of $K$. This is because $U(K)$ represents
$D_{+}(K,0)$ in the manifold obtained from the $+1$-surgery, and
the cabling operation converts this knot into the $(q,1)$-cable of
$D_{+}(K,0)$. Since untwisted Whitehead doubles are topologically
(but not necessarily smoothly) slice \cite{freedman_topology_2014},
$D_{+}(K,0)$ is topologically concordant to the unknot. It is well-known
that, if $K$ is concordant to $J$, then $K_{m,n}$ is concordant
to $J_{m,n}$ for all integers $m,n$. Hence $S_{q,1}$ is
also topologically concordant to the unknot $U_{q,1}$, and therefore
$S_{p}$ is topologically slice for all $p$.\

Now, define $S_{p}^{q}:=\#^{p}D_{+}(K,0)$. It is well-known that connected sums of topologically slice knots are topologically slice, hence $S_{p}^{q}$ is topologically slice. Moreover, as above, we have that $tu_{q}(S_{p}^{q})\leq p\cdot tu_{q}(S^{q})=p$.\

We now would like to get a lower bound on $u(S_{p}^{q})$ and thus to
show that $u(S_{p}^{q})-tu_{q}(S_{p}^{q})$ can be arbitrarily large. The
Ozsv\'{a}th-Szab\'{o} $\tau$ invariant gives such a lower bound. Thus, we
need to compute $\tau(S_{p}^{q})$ for all $p,q$. 

We show that $\epsilon(D_{+}(K,0))=1$ and hence, applying Theorem
\ref{thm:(Hom-2014)-Let}, that 
\[
\tau(S^{q})=q\tau(D_{+}(K,0)).
\]

We first compute $\tau(D_{+}(K,0))$. 
\begin{thm}
\label{thm:(Hedden-2006,-)}\cite{hedden_knot_2007} Let $D_{+}(K,t)$
denote the positive $t$-twisted Whitehead double of a knot $K$.
Then
\[
\tau(D_{+}(K,t))=\begin{cases}
1, & t<2\tau(K)\\
0 & \text{otherwise.}
\end{cases}
\]

\end{thm}
Since $\tau(K)>0$ in our case, $t=0<2\leq2\tau(K)$, and so $\tau(D_{+}(K,0))=1$.
Furthermore, as is the case with any Whitehead double, $g(D_{+}(K,0))=1$,
so $|\tau(D_{+}(K,0))|=1=g(D_{+}(K,0))$ and, by Proposition \ref{prop:(Hom-2014)-Let},
\[
\epsilon(D_{+}(K,0))=\text{sgn }\tau(D_{+}(K,0))=+1.
\]

We then apply Theorem \ref{thm:(Hom-2014)-Let} to $S^{q}$ to get
that
\[
\tau(S^{q})=q\tau(D_{+}(K,0)).
\]
Since $\tau(D_{+}(K,0))=1$, we have that $\tau(S^{q})=q$ and, hence,
$\tau(S_{p}^{q})=pq$. Thus,
$u(S_{p}^{q})\geq pq$. Therefore,
\[
u(S_{p})-tu_{q}(S_{p})\geq pq-p=p(q-1)\geq p,
\]
as desired.
\begin{acknowledgement*}
I would like to thank my adviser Tim Cochran for his invaluable mentorship
and support. Thanks also to Stefan Friedl, Maciej Borodzik, Peter
Horn, and Mark Powell for their mentorship, and to Ina Petkova for suggesting that cables of the trefoil would have unknotting number arbitrarily larger than their untwisting number.
\end{acknowledgement*}

\bibliographystyle{amsalpha}
\addcontentsline{toc}{section}{\refname}\bibliography{The-untwisting-number-of-a-knot}

\end{document}